\documentclass[12pt,reqno]{amsart}

\usepackage{latexsym}
\usepackage{amsmath}
\usepackage{amsfonts}
\usepackage{amssymb}
\usepackage{hyperref}
\usepackage{tikz}
\usepackage[mathscr]{euscript}

\usepackage{enumitem}

\newtheorem{theorem}{Theorem}[section]
\newtheorem{lemma}[theorem]{Lemma}
\newtheorem{corollary}[theorem]{Corollary}
\newtheorem{proposition}[theorem]{Proposition}

\newtheorem{sublemma}{}[theorem]

\theoremstyle{definition}

\theoremstyle{remark}

\numberwithin{equation}{section}

\newcommand{\A}{\mathscr{A}}
\newcommand{\B}{\mathscr{B}}

\newcommand{\I}{\mathscr{I}}
\newcommand{\C}{\mathscr{C}}
\newcommand{\1}{(i)}
\newcommand{\2}{(ii)}
\newcommand{\3}{(iii)}

\newcommand{\5}{(v)}
\newcommand{\M}{M(E,\A,c)}
\newcommand{\cl}{{\rm cl}}
\newcommand{\si}{{\rm si}}

\newcommand{\e}{\backslash e}

\newcommand{\Utfi}{U_{2,5}}
\newcommand{\Uthfi}{U_{3,5}}

\begin{document}

\title{Laminar Matroids}

\author{Tara Fife}
\email{fi.tara@gmail.com}
\author{James Oxley}
\email{oxley@math.lsu.edu}
\address{Mathematics Department, Louisiana State University, Baton Rouge, Louisiana, USA}

\date{\today}

 \begin{abstract}
 
A laminar family is a collection $\A$ of subsets of a set $E$ such that, for any two intersecting sets, one is contained in the other. For a capacity function $c$ on $\A$, let $\I$ be 
$\{I:|I\cap A| \leq c(A)\text{ for all $A\in\A$}\}$. Then $\I$ is the collection of independent sets of a (laminar) matroid on $E$.  We present a method of compacting laminar presentations, characterize the class of laminar matroids by their excluded minors, present a way to construct all laminar matroids using basic operations, and   compare the class of laminar matroids to other well-known classes of matroids.

\end{abstract}

\maketitle

\section{Introduction}
\label{Preliminaries}

The matroid terminology used here will follow Oxley \cite{James}. We start with some definitions not found in \cite{James}. Given a set $E$, a family $\A$ of subsets of $E$ is \textit{laminar} if, for every two sets $A$ and $B$ in $\A$ with $A\cap B\not=\emptyset$, either $A\subseteq B$ or $B\subseteq A$. Let $\A$ be a laminar family of subsets of a finite set $E$. Let $c$ be a function from $E$ into the set of real numbers. Define $\I$ to be the set of subsets $I$ of $E$ such that $|I\cap A| \leq c(A)$ for all $A$ in $\A$. It is well known (see, for example, \cite{LF,GK, IW, kl}) and easily checked that $\I$ is the set of independent sets of a matroid on $E$. We call $c$ a \textit{capacity function} for the matroid $(E,\I)$ and write this matroid as $M(E,\A,c)$. A matroid $M$ is \textit{laminar} if it is isomorphic to $M(E,\A,c)$ for some set $E$, laminar family $\A$, and capacity function $c$. We call $(E,\A,c)$ a {\it presentation} for $M$.

Laminar matroids have appeared quite frequently in the literature during the last fifteen years. Interest in them has focused  on how certain optimization problems,  particularly the matroid secretary problem, behave for such matroids \cite{bik, CL, FSZ,IW,JSZ, Soto}. This work was reviewed  by Huynh \cite{huynh}. With the exception of the thesis of Finkelstein~\cite{LF}, there appears to have been little work done on exploring the matroid properties of the class of laminar matroids. Here we do just that.  In particular, we give three characterizations of this class   of matroids beginning with  the following. 

\begin{theorem}
\label{Converse}
A matroid is laminar if and only if, for all circuits $C_1$ and $C_2$  with $C_1\cap C_2\not=\emptyset$, either $\cl(C_1)\subseteq \cl(C_2)$, or $\cl(C_2)\subseteq \cl(C_1)$. 
\end{theorem}

As we shall see, it is not difficult to show that the class 
of laminar matroids is minor-closed. For each $r \ge 3$, let $Y_r$ be the  matroid that is obtained by truncating, to rank $r$, the parallel connection of two $r$-element circuits.  From the last result, $Y_r$ is not laminar. Indeed, the collection of such matroids is the set of excluded minors for the class of laminar matroids. 

\begin{theorem}
\label{exmin}
A matroid is laminar if and only if it has no minor isomorphic to any member of $\{Y_r: r \ge 3\}$.
\end{theorem}

A transversal matroid is {\em nested} if it has a nested  presentation, that is, a transversal presentation $(B_1,B_2,\dots B_n)$  such that $B_1\subseteq B_2\subseteq\dots\subseteq B_n$. These matroids were introduced by Crapo \cite{crapo} and have appeared under a variety of names including freedom matroids \cite{CS}, generalized Catalan matroids \cite{BDMN}, shifted matroids \cite{ard} and Schubert matroids \cite{soh}  (see \cite{LPSP}). As we shall show, all nested matroids are laminar. 
Oxley, Prendergast, and Row \cite{Don} showed that the  class of nested matroids is minor-closed and they determined the excluded minors for this class. As the reader will observe, these excluded minors are strikingly similar to the excluded minors for the class of laminar matroids.

\begin{theorem}
\label{exnested}
A matroid is nested if and only if, for all $r \ge 2$, it has no minor isomorphic to the matroid that is obtained by truncating, to rank $r$, the direct sum of two $r$-element circuits. 
\end{theorem}

Crapo~\cite{crapo} showed that nested matroids coincide with the class of matroids that can be obtained from the empty matroid by applying the operations of adding a coloop and taking a free extension (see also \cite[Theorem 3.14]{BDMN}). A straightforward modification of this result yields the following  characterization of  nested matroids. 

\begin{theorem}
\label{nestedcons}
The class of nested matroids coincides with the class of matroids that can be obtained from the empty matroid by adding coloops and truncating.
\end{theorem}

Our third characterization  of the class of laminar matroids is a constructive one that reveals how nested matroids and laminar matroids differ.

\begin{theorem}
\label{construct}
The class of laminar matroids coincides with the class of matroids that can be constructed by beginning with the empty matroid and using the following operations. 
 \begin{itemize}
 \item[(i)] Adding a coloop to a previously constructed matroid.
 \item[(ii)]  Truncating a previously constructed matroid. 
 \item[(iii)] Taking the direct sum of two previously constructed matroids.   
 \end{itemize} 
\end{theorem}

In the next section, we prove Theorem~\ref{Converse} and show that every laminar matroid has a unique presentation with no superfluous information. In Section~\ref{EM}, we prove Theorem~\ref{exmin}, while, in Section~\ref{Construction}, we prove Theorem~\ref{construct} and determine all of the laminar matroids whose duals are also laminar.   Finkelstein~\cite{LF} showed that all laminar matroids are gammoids. Hence, by a result of \cite{PifWel}, all laminar matroids are representable over all sufficiently large fields. 
In Section~\ref{Boring}, we characterize binary laminar matroids and ternary laminar matroids.

\section{Canonical Presentation}
\label{MinPres}

In this section, we obtain a presentation for a laminar matroid that has no redundant information. It is clear that, when dealing with a capacity function $c$ of a laminar matroid, we lose no generality in assuming that the range of $c$ is the set of non-negative integers. The following lemma is an immediate consequence of the definition of laminar matroids.

\begin{lemma}
\label{independant}
If $I$ is independent in $M(E,\A,c)$ and $A \in \A$, then $I$ is independent in $M(E,\A-\{A\},c|_{\A-\{A\}})$. 
\end{lemma}

Throughout this section, we shall assume that $M$ is 
the laminar matroid $M(E,\A,c)$. Here and throughout the paper, whenever we write $c(A)$ it will be implicit that $A \in \A$. 
We say that a set $A\in\A$ is {\em essential} if $\M\not=M(E,\A-\{A\},c|_{\A - \{A\}})$. When $M$ has no loops,   we say that $(E,\A,c)$ is a {\em canonical presentation} for $M$ if every $A\in\A$ is 
essential. When $M$ has a loop, we say that a presentation of $M$ is  {\em canonical} if it can be written as $(E,\A\cup\{A_0\},c)$, where 
$A_0=\cl(\emptyset)$ and $(E-A_0,\A,c|_{\A})$ is a canonical presentation of $M\backslash A_0$. 

We omit the proof of the following  well-known observation (see, for example, \cite[Section 2.4]{kl}).

\begin{lemma} 
\label{BadChild}
Suppose $A$ and $B$ are members of $\A$  such that $B\subsetneqq A$ and $c(B)\geq c(A)$. Then $B$ is not essential.
\end{lemma}

Let $A$ and $H$ be members of a laminar family $\A$ and suppose that $A \subsetneqq H$. If  there is no $G\in\A $ such that $A\subsetneqq G \subsetneqq H $, then we say that $A$ is a {\em child} of $H$. 
For $A$ in $\A$, denote by $\chi(A)$ the set of children of $A$, and let $S(A)=\{e:e\in A-\cup_{F\in \chi(A)}F\}$. Observe that, in $M|A$, either all of the elements of $s(A)$ are coloops or all such elements are free. 

We define $b(A)=|S(A)|+\sum_{F\in \chi(A)}c(F)$. When  $A$ is essential, we now bound the capacity of $A$ in terms of $b(A)$. 

\begin{lemma}
\label{BadParent}
If $c(A)\geq b(A)$, then $A$ is not essential. 
\end{lemma}

\begin{proof}
Let $I$ be independent in $M(E,\A-\{A\},c|_{\A-\{A\}})$. Then $|I\cap F|\leq c(F)$ for all $F\in \chi(A)$, and $|I\cap S(A)|\leq |S(A)|$. Since the set $S(A)$  together with children of $A$  partitions   $A$, we see that 
$$|I\cap A| =   |I \cap S(A)|~+\sum_{F\in\chi(A)}|I \cap F|\quad\leq\quad  |S(A)|~+\sum_{F\in\chi(A)}c(F)=b(A).$$
Hence $A$ is not essential.
\end{proof}

The last lemma generalizes the following elementary fact about canonical presentations.
We omit the proof.

\begin{corollary} 
\label{obvious} 
If $(E,\A,c)$ is a canonical presentation for $M$, then $|A| > c(A)$ for all $A$ in $\A$.
\end{corollary}

With the goal of showing the uniqueness of  canonical presentations, next we exhibit some relationships between circuits and canonical presentations.
In particular, the next lemma will show that if $M$ has no loops, then  $\cl(C)\in\A$ for each circuit $C$ and $c\left(\cl(C)\right)=|C|-1$.

\begin{lemma}
\label{CircuitBehavior}
Let $C$ be a circuit of $M$. Assume that $(E,\A,c)$ is canonical.  Then 
\begin{enumerate}[label=(\roman*)]
\item $\A$ contains a member $A_C$ of capacity $|C|-1$ such that $C\subseteq A_C$; and
\item if $|C|\geq 2$, then $A_C=\cl(C)-\cl(\emptyset)$.
\end{enumerate}
\end{lemma}

\begin{proof}
Part (i) holds if $|C|=1$. Assume that $|C|\geq 2$, and that $e\in C$. Then, since $C$ is dependent, but $C-e$ is independent, we must have $e\in A$ for some $A\in \A$ where $|(C-e)\cap A|\leq c(A) $, but $|C\cap A|> c(A)$. Then $c(A)=|(C-e)\cap A|$. Now, $C\cap A$ is dependent, since $|C\cap A|> c(A)$. Thus $C\cap A=C$, so $C\subseteq A$ and $c(A)=|C|-1$. Hence (i) holds.

To prove (ii), assume that $|C|\geq 2$.
Let $f$ be an element of $\cl(C)-\cl(\emptyset)$. By (i), $C \subseteq A_C$. 
Suppose  $f\in \cl(C)-C$. Then there is some circuit $D$ with $f\in D\subseteq C\cup f$. Then, by (i), $D\subseteq A_D\in \A$ and $|D|-1=c(A_D)$. Since $f$ is not a loop, $C\cap D$, and hence $A_C\cap A_D$, is non-empty. As $\A$ is a laminar family, this implies that $A_C\subseteq A_D$, or $A_D\subseteq A_C$. But, since $c(A_D)=|D|-1\leq |C|-1 = c(A_C)$, we deduce, from Lemma~\ref{BadChild}, that $A_D\subseteq A_C$. Thus $f\in A_C$ as desired. Hence $\cl(C)-\cl(\emptyset)\subseteq A_C$.

Now, suppose that $f\in A_C-C$. Since $f \in A_C$, by the definition of a canonical presentation, $f \not\in \cl(\emptyset)$. 
 Arbitrarily choose an element $e$ of $C$. Then, since $|(C-e)\cup f|=|C|> c(A_C)$ and $(C-e)\cup f\subseteq A_C$, we have that $(C-e)\cup f$ is dependent, so $f\in \cl(C)$. 
\end{proof}

This lemma has the following consequence.
\begin{corollary}
\label{NiceCondition}
If $C$ and $D$ are intersecting circuits of an arbitrary matroid $N$ such that $\cl(C)\not\subseteq \cl(D)$ and $\cl(D)\not\subseteq \cl(C)$, then $N$ is not laminar.
\end{corollary}

\begin{proof}
Clearly neither $C$ nor $D$ is a loop. 
Assume that $N$ is laminar and let $(E,\A,c)$ be a canonical presentation of $N$. Since $C$ meets $D$, we deduce that  $A_C$ meets $A_D$. But neither is a subset of the other.
\end{proof}

\begin{theorem}
\label{Uniqueness}
A laminar matroid $M$  has a unique canonical presentation. Indeed, when $M$ is loopless, $\A =   \{\cl(C):C\text{ is a circuit of $M$}\}$ and $c(\cl(C)) = r(C) =  |C| - 1$. 
\end{theorem}

The core of the proof of this theorem is contained in the next result. 

\begin{lemma}
\label{*}
Let $(E,\A,c)$ be a canonical  presentation for a loopless laminar matroid $M$. If $A \in \A$, then $A$ is dependent. Moreover, if $C$ is a maximum-sized circuit contained in $A$, then $A_C = A$ so $c(A)=|C|-1$.
\end{lemma}

\begin{proof}
By Corollary~\ref{obvious}, $c(A) < |A|$. Thus $A$ is dependent. Now choose $A$ to be a minimal counterexample to \ref{*}. As $A \cap A_C \supseteq C$, either $A \subsetneqq A_C$ or $A_C \subsetneqq A$. In the first case, by Lemma~\ref{BadChild}, $c(A) \le c(A_C) - 1 = |C| - 2$. Hence $A$ cannot contain an independent set of size $|C| - 1$. This is a contradiction since $C \subseteq A$. Thus $A_C \subsetneqq A$. Now 
 $A$ has a child $A'$ such that $A_C\subseteq A'\subsetneqq A$. The  choice of $A$ implies that $A'=A_C$.

Let $A_1, A_2, \dots , A_n$ be the children of $A$ other than $A_C$ and write $A_0$ for $A_C$. Then, for each $i$, our choice of $A$ means  that $c(A_i)=|C_i|-1$, where $C_i$ is a maximum-sized circuit contained in $A_i$. Arbitrarily choose $e_i$ in $C_i$. Then   $C_i-e_i$ is a basis for $A_i$. Clearly $S(A)=A-(A_0\cup A_1\cup \dots \cup   A_n)$. 

By Lemma~\ref{BadParent}, $|S(A)| + \sum_{i=0}^{n}c(A_i) =b(A)\ge c(A)+1$. Now 
$|\bigcup_{i=0}^{n}(C_i-e_i)\cup S(A)| = b(A)$ so $\bigcup_{i=0}^{n}(C_i-e_i)\cup S(A)$ contains a subset $X$  such that $|X| = c(A) +1$. As $|X \cap A| = |X| > c(A)$, we see that $X$ is dependent. Thus $X$ contains a circuit $Z$ and $\cl(Z) = A_Z$. Then $c(A_Z) = |Z| - 1 \le |X| - 1 = c(A).$ As $A_Z$ and $A$ meet, it follows by Lemma~\ref{BadChild} that $A_Z\subseteq A$. Now $Z \not \subseteq S(A)$, otherwise $A_Z$ is a proper subset of $A$ that is in $\A$ but is not contained in a child of $A$. Thus either $Z$ meets $C_i$ and $C_j$ for some distinct $i$ and $j$, or $Z$ meets $C_i$ and $S(A)$. In each case, 
by Corollary~\ref{NiceCondition}, $\cl(C_i) \subseteq \cl(Z)$, or $\cl(Z) \subseteq \cl(C_i)$. If $Z$ meets $S(A)$, then $Z \not \subseteq \cl(C_i)$ so $\cl(C_i) \subsetneqq \cl(Z)$. 
The last inclusion also holds if $Z$ meets $C_j$ since $\cl(C_i)$ and $\cl(C_j)$ are disjoint.  
As $\cl(C_i)$ is a child of $A$, and $\cl(C_i) \subseteq \cl(Z) \subsetneqq A$, it follows that $\cl(Z) = \cl(C_i)$; a contradiction.
\end{proof}

\begin{proof}[Proof of Theorem~\ref{Uniqueness}.] 
 It suffices to prove the result for loopless matroids. Suppose that $(E,\A,c)$ is a canonical presentation for $M$, and let $M$ be loopless. Then $\A\supseteq \{A_C:C\text{ is a circuit of $M$}\}$. Now take $A$ in $\A$. Then, by Lemma~\ref{*}, $A = A_C$ where $C$ is a maximum-sized circuit contained in $A$. Thus the theorem holds.
 \end{proof}

We omit the proof of the following elementary result (see, for example, \cite[Exercise 1.1.5]{James}).

\begin{lemma}
\label{TwoCircuitClosure}
Let $N$ be a matroid, $C$ be a circuit of $N$, and $e$ be a non-loop element of $\cl(C)-C$. Then $N$ has circuits $D$ and $D'$  such that $e\in D\cap D'$ and $(D\cup D')-e=C$.
\end{lemma}

Next we prove our first main result.

\begin{proof}[Proof of Theorem~\ref{Converse}.]
By Corollary~\ref{NiceCondition}, if $C$ and $D$ are  intersecting circuits in a  laminar matroid, then $\cl(D) \subseteq \cl(C)$ or $\cl(C) \subseteq \cl(D)$. 
 To prove the converse,  let $N$ be a matroid in which, for every two intersecting circuits, the closure of one is contained in the closure of the other. We may assume that $N$ is 
 loopless. Let $\A'=\{\cl(C):C\in\C(N)\}.$

Suppose $A_1,A_2\in\A'$ and $A_1\cap A_2\not=\emptyset$. Let $C_1$ and $C_2$ be  circuits so that $\cl(C_i)=A_i$ for each $i$ in $\{1,2\}$. If $C_1\cap C_2\not=\emptyset$, then, by the given condition, $A_1\subseteq A_2$ or $A_2\subseteq A_1$. Now suppose that $C_1\cap C_2=\emptyset$ and $e\in \cl(C_1)\cap \cl(C_2)$. Since $e$ is not a loop, Lemma~\ref{TwoCircuitClosure} implies that, for each $i$ in $\{1,2\}$, there are circuits $D_i$ and $D_i'$ of $M$ such that $e  \in D_i  \cap D_i'$ and $D_i \cup D_i' = C_i \cup e$. Now $e \in D_1\cap D_2$  so our hypothesis implies, without loss of generality, that $\cl(D_1)\subseteq \cl(D_2)$. 

If $\cl(D_1')$ is contained in either $\cl(D_2)$ or $\cl(D_2')$, then $C_1$ and hence $\cl(C_1)$ is contained in $\cl(C_2)$. But otherwise,  both $\cl(D_2)$ and $\cl(D_2')$ are subsets of $\cl(D_1')$, so $\cl(C_2)\subseteq \cl(C_1)$. We conclude that $\A'$ is a laminar family.  

For each $A$ in $\A'$, let $c'(A)=r_N(A)$, and let $N'=M(E,\A',c')$. 
We shall show that every circuit of $N$ is dependent in $N'$, and every   circuit of $N'$ is dependent in $N$. From this, it will follow immediately that $N = N'$ (see, for example, \cite[Lemma 2.1.22]{James}). 
Suppose $C$ is a circuit of $N$. Then $|C\cap \cl(C)|=|C|> r_N(C) = c'(\cl(C))$, so $C$ is dependent in $N'$. Now let $D$ be a circuit of $N'$. Then $\A'$ contains a set $A'$ such that $c'(A') < |D \cap A'|$. But, for all $d$ in $D$, as $D- d$ is independent in $N'$, it follows that $|(D- d) \cap A'| \le c'(A')$. Hence $D \subseteq A'$ and $c'(A') < |D|$. But $c'(A') = r_N(\cl(C'))$ for some circuit $C'$ of $N$. Thus $D$ is dependent in $N$ otherwise $c'(A') \ge |D|$; a contradiction. 
We conclude that $N = N'$, so $N$ is laminar and the theorem holds.
\end{proof}

The next two results are immediate consequences of Theorem~\ref{Converse}. 

\begin{corollary}
\label{OnlyCheckNonSpanningCircuit}
A matroid   is laminar if and only if, for every pair $C_1, C_2$ of non-spanning circuits with $C_1\cap C_2\not=\emptyset$, either $\cl(C_1)\subseteq \cl(C_2)$ or $\cl(C_2)\subseteq \cl(C_1)$.
\end{corollary}

\begin{corollary}
\label{AtMostOneNonSpanningCircuit}
Every matroid  with at most one non-spanning circuit is laminar.            
\end{corollary}

Our third corollary of Theorem~\ref{Converse} requires some more proof. 

\begin{corollary}
\label{conn}
Let $M$ be a loopless laminar matroid and $(E,\A,c)$ be its canonical presentation. Suppose $|E| \ge 2$. Then the following are equivalent.
\begin{itemize}
\item[(i)] $M$ is connected;
\item[(ii)] $E \in \A$; and
\item[(iii)] $M$ has a spanning circuit.
\end{itemize} 
\end{corollary}

\begin{proof} 
Since $M$ is loopless, clearly (iii) implies (i). 
Suppose  $E \in \A$. Then, by Lemmas~\ref{CircuitBehavior} and \ref{*}, as $M$ is loopless, $E = A_C = \cl(C)$  where $C$ is a maximum-sized circuit of $M$. Hence (ii) implies (iii).

Finally, suppose $M$ is connected   but $E \not\in \A$. If $M$ has an element $e$ that is in no member of $\A$, then  $e$ is a coloop of $M$; a contradiction.  Thus, if $F_1,F_2,\dots,F_k$ are the maximal members of $\A$, then $F_1 \cup F_2 \cup \dots \cup F_k = E$ We shall show that $k = 1$. Assume  $k >1$. For each $i$ in $\{1,2,\dots,k\}$, let $C_i$ be a maximum-sized circuit contained in $F_i$. Then, by Lemmas~\ref{CircuitBehavior} and \ref{*}, $F_i = \cl(C_i)$. As $M$ is connected, it has a circuit $D$ meeting $C_1$ and $C_2$. By Corollary~\ref{NiceCondition}, either $\cl(D) \subseteq \cl(C_i)$ for  some $i$ in $\{1,2\}$, or $\cl(D)$ contains both $F_1$ and $F_2$. Since $F_1$ and $F_2$ are disjoint, the latter holds. By Theorem~\ref{Uniqueness}, $A_D \in \A$, and $A_D$ contains $F_i$ and $F_2$; a contradiction. Hence $E \in \A$. We conclude that (i) implies (ii), so the corollary holds. 
\end{proof}

The next result follows immediately from the last result and Theorem~\ref{Uniqueness}.

\begin{corollary}
\label{ConnectedFlats}
Let $M$ be a loopless laminar matroid and $(E,\A,c)$ be its canonical presentation. Then the members of $\A$ are connected flats of $M$.
\end{corollary}

\begin{corollary}
\label{ConnectedFlats2}
In a laminar matroid with $(E,\A,c)$ as its canonical presentation, if $F$ is a connected flat of $M$ with $|F|\geq 2$, then $F\in\A$. 
\end{corollary}
\begin{proof}
Since $F$ is a connected flat of $M$, by Corollary~\ref{conn} $M|F$ has a spanning circuit $C$ in $M|F$. But $C$ is also a circuit of $M$, so 
$F=\cl(C)\in\A$.
\end{proof}

\section{Excluded Minors}
\label{EM}

In this section, we show that  the class of laminar matroids is a minor-closed,  and we prove our excluded-minor characterization.

\begin{lemma}
\label{MinorClosed}
Every minor of a laminar matroid is laminar.
\end{lemma}

\begin{proof}
Let $M$ be a laminar matroid and $(E,\A,c)$ be its canonical    presentation. Suppose $e\in E$. Clearly $\{A-e:A\in\A\}$ is a laminar family; we denote it by  $\A - e$. Observe that, if $A$ and $A'$ are members of $\A$ with $A\subsetneqq A'$, then $|A'-A|\geq 2$. To see this, note that, by Lemmas~\ref{BadChild} and \ref{BadParent} and Corollary~\ref{obvious},   $c(A) + 2\le c(A') + 1\le b(A')\leq c(A)+|A'-A|$.

For each $A'\in \A-e$, choose the unique $A\in\A$ with $A-e=A'$, and let 
$ c'(A') = c(A)$.
We shall show that 

\begin{sublemma}
\label{deldet}
$M\e=M(E -e, \A - e,c')$. 
\end{sublemma}

Suppose that $I$ is independent in $M\e$. Then $|I \cap A| \le c(A)$ for all $A$ in $\A$. As $e \notin I$, it follows that 
$|I \cap (A-e)| \le c'(A-e)$ for all $A-e$ in $\A-e$. Thus $I$ is independent in $M(E -e, \A - e,c')$. 
 
Now, suppose that $J$ is independent in  $M(E -e, \A - e,c')$. Then $|J\cap (A-e)|\leq c'(A-e)$ for all $A-e\in\A-e$. Now, for each $A\in\A$, we have $|J\cap A|=|J\cap (A-e)|\leq c'(A-e)= c(A)$, so $J$ is independent in $M$ and, hence, is independent in $M\e$. We conclude that \ref{deldet} holds.
 
To show that $M/e$ is laminar, we 
  may assume that $e$ is not a loop  otherwise the result holds by \ref{deldet}. Now, define $c''$ on $\A-e$ by 
\[
 c'' (A-e) =
  \begin{cases} 
       c(A) - 1              & \text{ if $e \in A$}; \\
       c(A)       & \text{ if $e\not\in A$.} \\
  \end{cases}
\] We will show that

\begin{sublemma}
\label{concon}
$M/e=M(E-e,\A - e,c'')$. 
\end{sublemma}

Suppose that $I$ is independent in $M/e$. Then $I\cup e$ is independent in $M$. Thus, $|(I\cup e)\cap A|\leq c(A)$ for all $A\in\A$. 
Now 
\begin{align*}
|I\cap (A-e)|=|(I\cup e)\cap (A-e)|&=   \begin{cases}
 |(I\cup e)\cap A|-1  & \text{ if $e \in A$}; \\
  |(I\cup e)\cap A|   & \text{ if $e \not \in A$}; \\
  \end{cases}\\
 &\le   \begin{cases}
 c(A)-1  & \text{ if $e \in A$}; \\
 c(A)  & \text{ if $e \not \in A$}; \\
  \end{cases}\\
  & = c''(A-e).
 \end{align*}
Thus  $I$ is independent in $M(E-e,\A',c'')$.

Now suppose that $J$ is independent in $M(E-e,\A - e,c'')$. Then $|J\cap A'|\leq c''(A')$ for all $A'\in\A'$. Let $A\in\A$ be such that $A'=A-e$.
Then 
 \begin{align*}
|(J\cup e)\cap A|&=   \begin{cases}
 |J\cap (A-e)|+1  & \text{ if $e \in A$}; \\
 |J\cap (A-e)| & \text{ if $e \not \in A$}; \\
 \end{cases}\\
 &\le   \begin{cases}
 c''(A-e) + 1  & \text{ if $e \in A$}; \\
 c''(A-e)  & \text{ if $e \not \in A$}; \\
  \end{cases}\\
  & = c(A).
 \end{align*}
We conclude that $J\cup e$ is independent in $M$, so $J$ is independent in $M/e$. Thus \ref{concon} holds and, hence, so does the theorem.
\end{proof}

We now prove our second main result.

\begin{proof}[Proof of Theorem~\ref{exmin}.] 
Recall that, for each $r \ge 3$, the matroid $Y_r$ is obtained from the parallel connection of two $r$-element circuits $C_1$ and $C_2$ across the basepoint $p$ by truncating this parallel connection to rank $r$. Since the only new circuits created by truncation are spanning, it follows that $C_1$ and $C_2$ are the only non-spanning circuits of $Y_r$.  

\setcounter{theorem}{2}
\setcounter{sublemma}{0}

\begin{sublemma}
\label{sufficient}
$Y_r$ is an excluded minor for the class of laminar matroids for all $r\geq 3$.
\end{sublemma}

To see this, first observe that, by Corollary~\ref{OnlyCheckNonSpanningCircuit}, $Y_r$ is not laminar.  Let $e\in E$. Without loss of generality, we may assume that $e\in C_1$.   The only potential non-spanning circuit of $M\e$ is $C_2$. Thus $M\e$ is laminar by 
Corollary~\ref{AtMostOneNonSpanningCircuit}. 

The contraction $M/p$ has $C_1 - p$ and $C_2 - p$ as its only  non-spanning circuits. Since these circuits are disjoint, it follows by  
Corollary~\ref{OnlyCheckNonSpanningCircuit} that $M/p$ is laminar. Now assume that $e\not=p$ and $e\in C_1$. Then $M/e$ has $C_1-e$ as its  only non-spanning circuit. Thus $M/e$ is laminar by Corollary~\ref{AtMostOneNonSpanningCircuit}. We conclude that \ref{sufficient} holds.

Now let $N$ be an excluded minor for the class of laminar matroids. 
Since $N$ is not laminar, by Theorem~\ref{Converse}, $N$ contains two intersecting   circuits $C_1$ and $C_2$ such that $C_1\cap C_2\not=\emptyset$ and neither $\cl(C_1)\subseteq \cl(C_2)$ nor $\cl(C_2)\subseteq \cl(C_1)$. Choose such a pair of circuits $\{C_1,C_2\}$ such that $|C_1\cup C_2|$ is minimal.

Since $N$ is an excluded minor, $|C_1\cap C_2|=1$ otherwise, for $e$ in $C_1 \cap C_2$, the matroid $N/e$ has $C_1 - e$ and $C_2 - e$ as intersecting circuits with the closure of neither containing the other, so $N/e$ is not laminar; a contadiction.
Similarly, $E(N)=C_1\cup C_2$ otherwise deleting an element of $E(N)-(C_1\cup C_2)$ would yield a non-laminar matroid. 

\begin{sublemma}
\label{Spanning}
Let $C$ be a circuit of $N$ such that $C$ meets $C_1-\cl(C_2)$ and $C_2-\cl(C_1)$. Then $C$ is spanning.
\end{sublemma}

To see this, note that, as $C\not\subseteq \cl(C_1)$ and $C\not\subseteq \cl(C_2)$, Theorem~\ref{Converse} implies that  $C_1\subseteq \cl(C)$ and $C_2\subseteq \cl(C)$. Hence  $C$ is spanning.

\begin{sublemma}
\label{ClosureOfC1}
$\cl(C_i)=C_i$ for each $i$ in $\{1,2\}$.
\end{sublemma}

It suffices to prove this assertion for $i = 1$. Suppose $e\in \cl(C_1)-C_1$. By Lemma~\ref{TwoCircuitClosure}, $N$ has circuits $D$ and $D'$ with $e\in D\cap D'$ and $C_1\cup e= D\cup D'$. Both $C_1-D$ and $C_1-D'$ are non-empty so $|C_1 \cup C_2|$ exceeds both $|C_2 \cup D|$ and $|C_2 \cup D'|$. Hence, by the minimality assumption, either $\cl(C_2)$ is contained in one of $\cl(D)$ or $\cl(D')$; or $\cl(C_2)$ contains both $\cl(D)$ and $\cl(D')$. This gives a contradiction since, in the first case, $\cl(C_2)\subseteq \cl(C_1)$ while,   in the second,  $\cl(C_1)\subseteq \cl(C_2)$. Thus  $\cl(C_1)=C_1$.

\begin{sublemma}
\label{cranky}
$|C_1| = r(N) = |C_2|$.
\end{sublemma}

Take $e$ in $C_1 - C_2$. As $\cl(C_2) = C_2$, it follows that $C_1 - e$ and $C_2$ are intersecting circuits of $N/e$. Hence, by Theorem~\ref{Converse}, $\cl_{N/e}(C_1 - e) \supseteq \cl_{N/e}(C_2)$, or $\cl_{N/e}(C_2)  \supseteq  \cl_{N/e}(C_1 - e)$. The first possibility gives the contradiction that $\cl_N(C_1) \supseteq \cl_N(C_2 \cup e) \supseteq \cl_N(C_2)$. Hence $\cl(C_2 \cup e) \supseteq \cl(C_1)$, so $\cl(C_2 \cup e) = E(N)$. Thus $C_2 \cup e$ spans N while the circuit $C_2$ does not, so $r(N) = |C_2|$. By symmetry, \ref{cranky} holds.

\begin{sublemma}
\label{NonSpanningCircuits}
The only non-spanning circuits of $N$ are $C_1$ and $C_2$.
\end{sublemma}

Let $D$ be a non-spanning circuit of $N$ that differs from $C_1$ and $C_2$. Then $D$ meets each of $C_1-C_2$ and $C_2-C_1$. As $\cl(C_2)=C_2$ and $\cl(C_1)=C_1$, we deduce by \ref{Spanning} that $D$ is spanning. Thus \ref{NonSpanningCircuits} holds. 

Since $\cl(C_1) = C_1$, we deduce that $r(N) \ge 3$. Recalling that a matroid of given rank is uniquely determined by a list of its non-spanning circuits (see, for example, \cite[Proposition 1.4.14]{James}), we deduce that $N \cong Y_r$ for some $r\ge 3$. 
\end{proof}

The following is an immediate consequence of Theorem~\ref{exmin}.

\begin{corollary}
\label{rsmall}
Every matroid of rank at most two is laminar.
\end{corollary}

\section{Constructing Laminar Matroids}
\label{Construction}

In this section, we begin by proving our third characterization of laminar matroids, Theorem  \ref{construct}. We then show that all nested matroids are laminar and we determine precisely which laminar matroids have duals that are also laminar.

\begin{proof}[Proof of Theorem~\ref{construct}.]
We first show that the class of laminar matroids is closed under adding coloops, truncating, and taking direct sums. If $\M$ is a laminar matroid, then we see that  $\M\oplus U_{1,1} = M(E\cup e,\A,c)$. Further, when $r(\M) > 0$, one easily checks that $T(\M) = M(E,\A',c')$ where $\A'=\A\cup\{E\}$, while $c'(E)=r_M(E)-1$, and $c'(A)=c(A)$, for all $A$ in $\A' - \{E\}$. Finally,  let  $(E_1,\A_1,c_1)$ and $(E_2,\A_2,c_2)$ be canonical presentations for laminar matroids on disjoint sets $E_1$ and $E_2$. Then $M(E_1,\A_1,c_1)\oplus M(E_2,\A_2,c_2)=M(E_1\cup E_2, \A_1 \cup \A_2, c)$, where $c$ coincides with $c_1$ when restricted to $\A_1$ and   with $c_2$ when restricted to $\A_2$.

Let $M$ be a laminar matroid having $(E,\A,c)$ as its canonical presentation. To prove that every laminar matroid can be constructed from the empty matroid in the manner  described, we proceed by induction on $|E(M)|$.  
The result is immediate if $|E(M)| \le 1$. Assume it holds if $|E(M)| <k$ and let $|E(M)|  = k \ge 2$. 
 If $\M$ is disconnected, then $\M$ is the direct sum of its components, each of which can be constructed.  Hence we may assume that $M$ is connected. Thus $M$ is loopless. Moreover,  by Corollary~\ref{conn},  $E\in\A$. Let $A_1,A_2,\dots, A_n$ be the children of $E$ in $\A$. Then, by Theorem~\ref{Uniqueness}, each $A_i$ is a flat of $M$ and $c(A_i) = r(A_i)$. 
 
Suppose first that $E - \cup_{i =1}^n A_i$ is non-empty and let $e$ be in this set. Then $e$ is free in $M$. Now $M\backslash e$ can be constructed in the manner described. Since $M$ can be obtained from $M\backslash e$ by adjoining $e$ as a coloop and then trucating the resulting matroid, we deduce that $M$ can be constructed in the desired manner. We may now assume that $E = \cup_{i =1}^n A_i$. 
 
For each $i$ in $\{1,2,\dots,n\}$, 
 let $M_i$ be $M(A_i,\A_i,c_i)$, where     $\A_i=\A\cap 2^{A_i}$, and $c_i$ is the restriction of $c$ to $\A_i$. Evidently 
 \begin{equation}
 \label{restnew}
 M_i = M|A_i.
 \end{equation} 
 Since $E = \cup_{i =1}^n A_i$, it follows that $n \ge 2$. 
 
 We show next that, for all  $i$ in $\{1,2,\dots,n\}$, 
  \begin{equation}
 \label{low}
 r(M_i) < r(M) < r(M_1) + r(M_2) + \dots + r(M_n).
 \end{equation}
 
The first inequality follows by  Theorem~\ref{Uniqueness}, Lemma~\ref{BadChild}, and Corollary~\ref{conn} since $r(A_i) = c(A_i) <c(E) = r(M)$. 
The second inequality    is an immediate consequence of the fact that $M$ is connected. We conclude that (\ref{low}) holds.  

Now let $r = r(M)$ and let $M'$ be the truncation of $M_1 \oplus M_2 \oplus \dots \oplus M_n$ to rank $r$. Then, by (\ref{restnew}) and (\ref{low}), for all  $i$ in $\{1,2,\dots,n\}$,

 \begin{equation}
 \label{middle}
M|A_i= M_i = M'|A_i.
 \end{equation}

To complete the proof of the theorem, observe that the following are equivalent for a subset $X$ of $E$.  
\begin{itemize}
\item[(i)] $X$ is a non-spanning circuit of $M'$; 
\item[(ii)] $X$ is a non-spanning circuit of $M_1 \oplus M_2 \oplus \dots \oplus M_n$; 
\item[(iii)] $X$ is a  circuit of $M_i$ for some $i$ in $\{1,2,\dots,n\}$; 
\item[(iv)] $X$ is a non-spanning circuit of $M$.
\end{itemize}
Since $r(M') = r(M)$, we deduce that $M' = M$. 
\end{proof}

Now let $M$ be a nested matroid having $(B_1,B_2,\dots,B_n)$ as a presentation   with $\emptyset \neq B_1 \subsetneqq B_2 \subsetneqq \dots \subsetneqq B_n$. Let  $B_0 = E(M) - B_n$. Then one easily checks that $M = M(E,\B,c)$ where $\B = \{B_0,B_1,\ldots,B_n\}$ and $c(B_i) = i$ for all $i$. In fact, all nested matroids are laminar.

\begin{lemma}
\label{NestedMatroidAreLaminar}
Let $N$ be a transversal matroid having a nested presentation $(B_1,B_2,\dots B_n)$. Then $N$ is laminar with the set $\A$ in its canonical presentation consisting of $\cl(\emptyset)$ together  with the unique maximal subset of $\{B_1,B_2,\dots,B_n\}$ no two members of which are equal. 
\end{lemma}

\begin{proof} Clearly we may assume that each $B_i$ is non-empty. 
As above, let $B_0 = E(N) - B_n$ and take $c(B_0) = 0$. The members of $B_1,B_2,\dots,B_n$ need not be distinct. Pass through this list of sets deleting each $B_i$ for which $B_i = B_{i+1}$. Let $\B$ be the resulting collection of distinct sets $B_{i_1},B_{i_2},\dots,B_{i_k}$. Each of these sets is properly contained in its successor. Define $c(B_{i_t}) = r_N(B_{i_t})$. Then it is straightforward to check that $N = M(E(N), \B \cup \{B_0\},c)$. 
\end{proof}

\begin{corollary}
\label{NestedLaminar}
A loopless laminar matroid $M$ with canonical presentation $(E,\A,c)$ is nested if and only if $\A$ is  a chain  under inclusion.
\end{corollary}
\begin{proof}
If $M$ is nested, then it is an immediate consequence of the last lemma that $\A$ is indeed a chain. Conversely, 
let 
 $\A=\{A_0, A_1, \ldots, A_n\}$ where $A_i\subsetneqq A_{i+1}$ for all $i < n$. Let ${\mathscr B}$ be the family of sets consisting of $\cl(A_1)$ copies of $A_1$ along with  $c(A_i)-c(A_{i-1})$ copies of $A_i$ for all $i$ in  $S\{2,3,\dots,n\}$. The corollary now follows without difficulty. 
\end{proof}

The last result is reminiscent of the following result of \cite{Don}, which was elegantly restated by Bonin and de Mier~\cite{LPSP}. Recall that a {\it cyclic flat} of a matroid is a flat that is a union of circuits. 

\begin{lemma}
\label{nestedcyc}
A matroid is nested if and only if its collection of cyclic flats is a chain under inclusion. 
\end{lemma}

The next result  determines precisely which matroids have the property that both $M$ and $M^*$ are laminar, noting that nested matroids are a fundamental class of such matroids. 

\begin{proposition}
\label{newbie} 
The following are equivalent for a matroid $M$. 
\begin{itemize}
\item[(i)] Both $M$ and $M^*$ are laminar.
\item[(ii)] Each component of $M$ is either a nested matroid or is a truncation to some non-zero rank of the direct sum of two uniform matroids of positive rank.
\end{itemize}
\end{proposition}

The proof of this proposition will use the following. 

\begin{lemma}
\label{DualFlats}
Let $(E,\A,c)$ be the canonical presentation of a connected laminar matroid $M$. If $A\in\A$, then $E-A$ is a connected flat of $M^*$.
\end{lemma}
\begin{proof}
Suppose $e \in A$ and $e \in \cl^*(E-A)$. Then $e \not\in \cl(A- e)$. This gives a contradiction to Theorem~\ref{Uniqueness} as $A$ is the closure of a circuit of $M$. We deduce that $E-A$ is flat of $M^*$. 

Now assume $M^*|(E-A)$ is disconnected. Then so is $M/A$.  Let  $X$ and $Y$ be distinct components of $M/A$. As $M$ has no coloops,  each of $X$ and $Y$ has at least two elements.  
If both $X \cup A$ and $Y \cup A$  are connected flats of $M$, then, by Corollary~\ref{ConnectedFlats}, both are in $\A$ and we contradict the fact that $\A$ is laminar. Thus we may assume that $M|(X \cup A)$ is disconnected.  Then the last matroid has  $X$ and $A$ as its components.  Now $M$ has a circuit $C$ that meets both $A$ and $X$. This circuit must also meet $E-(X \cup A)$ otherwise $M|(X \cup A)$ is connected. Now $C- A$ is a union of circuits of $M/A$. Since no such circuit meets both $X$ and $E-(X \cup A)$, there is a circuit $D$ of $M/A$  contained in $C \cap X$. As $r(A) + r(X) = r(A \cup X)$, it follows that $D$ is a circuit of $M$ that is properly contained in $C$; a contradiction. We conclude that $E-A$ is a connected flat in $M^*$.
\end{proof}

\begin{proof}[Proof of Proposition~\ref{newbie}.] 
Since, by Theorem~\ref{construct}, the class of laminar matroids is closed under direct sums,  it suffices to prove the proposition in the case that $M$ is connected. If $M$ is nested, then, by \cite{Don}, so is $M^*$. Now suppose that $M$ is the truncation to some positive rank $r$ of the direct sum of uniform matroids $M_1$ and $M_2$ where $r(M_1) \ge r(M_2) > 0$.  Since each uniform matroid is laminar and the class of laminar matroids is closed under direct sums and truncation, $M$ is laminar. To see that $M^*$ is also laminar, suppose first that $r = r(M_1)$. Then $M$ has $E(M_2)$ as its unique proper non-empty cyclic flat. Hence $M^*$ has $E(M_1)$ as its unique proper non-empty cyclic flat so, by Lemma~\ref{nestedcyc}, $M^*$ is nested and hence is laminar. We may now assume that $r > r(M_1)$. For each $i$ in $\{1,2\}$, let $E_i = E(M_i)$. Then, by determining the hyperplanes of $M$, it follows  that the only non-spanning circuits of $M^*$ consist, for each permutation $(i,j)$ of  $\{1,2\}$ of those subsets of $E_i$ with exactly $|E_i| - r + r_j + 1$ elements. Thus $M^*|E_i$ is uniform for each $i$, and $M^*$ is obtained by truncating the direct sum of  $M^*|E_1$ and $M^*|E_2$ to rank $|E(M)| - r$. We deduce, from above, that $M^*$ is laminar. 

To prove the converse, let $M$ be a connected laminar matroid such that $M^*$ is laminar but $M$ is not nested. Let $(E,\A,c)$ be the canonical presentation of $M$. 
Since $M$ is not nested, $\A$ contains two disjoint sets, $A_1$ and  $A_2$. Without loss of generality, we may assume that $A_1$ and $A_2$ are minimal members of $\A$. Let $(E,\A^*,c^*)$ be the canonical presentation of $M^*$.
By Lemmas~\ref{DualFlats} and \ref{ConnectedFlats2}, $E-A_1$ and $E-A_2$ are in $\A^*$. Thus $E-A_1$ and $E-A_2$ are disjoint. Hence  $E=A_1\cup A_2$. Since the only non-spanning circuits of $M$  are contained in  $A_1$ or $A_2$, and $M|A_i$   is   uniform matroid for each $i$, the result follows.
\end{proof}

The last proof was inspired by the work of Bonin and de Mier~\cite{LPSP} on lattice path matroids. Elsewhere, we will describe the relationship between the class of such matroids and the class of laminar matroids.

\section{Which laminar matroids are binary or ternary?}
\label{Boring}

In this section, we determine precisely which laminar matroids are binary or ternary.  We begin by showing that the class of laminar matroids is closed under the operation of parallel extension.

\begin{lemma}
\label{pe}
Every parallel extension of a laminar matroid is laminar.
\end{lemma}

\begin{proof} Let $M$ be a laminar matroid having $(E,\A,c)$ as its  canonical presentation. By Theorem~\ref{construct}, it suffices to prove the result when $M$ is loopless. Let $e$ be an element of $M$ and let $f$ be an element not in $E$. Suppose first that $e$ is in a $2$-circuit $C$ of $M$. Then $A_C \in \A$. Now add $f$ to every member of $\A$ that contains $e$ leaving the capacity of each such set unchanged.  It is straightforward to check that this process yields a laminar matroid that is an extension of $M$ and has   $\{e,f\}$ as a circuit. If $e$ is not in a $2$-circuit of $M$, then add $\{e,f\}$ to $\A$ as a set of capacity $1$,  and add $f$ to every member of $\A$ that contains $e$ leaving the capacity of each such set unchanged. Again,§ it is straightforward to check that this gives a laminar matroid that is a parallel extension of $M$.
\end{proof}

The characterizations of the laminar matroids that are binary or ternary involve the matroid $Y_r$. We recall that, for $r \ge 3$,  this matroid   is the truncation to rank $r$ of the parallel connection,  across the basepoint $p$, of two $r$-element circuits.  In particular, $Y_3$ is isomorphic to a single-element deletion of $M(K_4)$. Moreover, for all $r$, 

\begin{equation}
\label{yr}
Y_r\backslash p \cong U_{r,2r-2}. 
\end{equation}

\begin{theorem}
\label{Binary}
The following are equivalent for a matroid $N$.
\begin{enumerate}[label=(\roman*)]
\item Each component of $N$ has rank at most one or can be obtained from a circuit by a sequence of parallel extensions.
\item $N$ is graphic and laminar.
\item $N$ is regular and laminar.
\item $N$ is binary and laminar.
\item $N$ has no minor in $\{Y_3,U_{2,4}\}$
\end{enumerate}
\end{theorem}

\begin{proof}
It is clear that each of (i)--(iv) implies its successor.  We complete this proof by showing that (v) implies (ii), and that (ii) implies (i).
To show that \5 implies \2, suppose that $N$ has no $Y_3$- or $U_{2,4}$-minor. Then $N$ has no minor in  $\{U_{2,4}, F_7, F_7^*, M^*(K_5), M^*(K_{3,3})\}$. Thus $N$ is graphic. Moreover,    by (\ref{yr}), $Y_r$ has a $U_{2,4}$-minor for all $r\geq 4$. It follows, by Theorem~\ref{exmin}, that $N$ is laminar. Thus  \5 implies \2.

To show that \2 implies \1, suppose that $N$ is a graphic laminar matroid. Then, by Theorem~\ref{construct} and Lemma~\ref{pe}, we may assume that $N = M(G)$ for some  simple 2-connected graph $G$ having at least three vertices.   Now $G$ does not have $K_4 \e$ as a minor. Take a maximal-length cycle $C$ in $G$. Then either $E(G) = E(C)$, or $G\backslash E(C)$ has a path joining distinct non-consecutive vertices of $C$. In the latter case,  $G$ has $K_4 \e$ as a minor. This contradiction completes the proof.
\end{proof}

To prove the characterization of ternary laminar matroids, we shall use a   result of Cunningham and Edmonds (in \cite{cunn}) that decomposes a $2$-connected matroid into circuits, cocircuits, and $3$-connected matroids. 
The terminology used here follows \cite[Section 8.3]{James}.

\begin{theorem}
\label{2tree2}
Let $M$ be a $2$-connected matroid. Then $M$ has a tree decomposition 
$T$ in which every vertex label is $3$-connected, a circuit, or a cocircuit, and there are 
no two adjacent vertices that are both labelled by circuits or are both labelled by cocircuits. 
Moreover, $T$ is unique to within relabelling of its edges.
\end{theorem}

The tree decomposition of $M$ whose uniqueness is guaranteed by the last theorem is called the {\it canonical tree decomposition} of $M$.

\begin{theorem}
\label{ternary} 
The following are equivalent for a matroid $N$.
\begin{enumerate}[label=(\roman*)]
\item $N$ is ternary and laminar.
\item $N$ has no minor in $\{\Utfi,\Uthfi, Y_3\}$.
\item Each component of $\si(N)$ has rank at most one, is   $U_{2,4}$  or $U_{2,4} \oplus_2 U_{2,4}$, or can be obtained from an $n$-circuit for some $n \ge 3$ by $2$-summing on copies of $U_{2,4}$ across $k$ distinct elements of the circuit for some $k$ in $\{0,1,\dots,n\}$.
\end{enumerate}
\end{theorem}

\begin{proof}
It is clear that \1 implies \2. Moreover, it follows from (\ref{yr}) that \2 implies \1. 

To show that \2 implies \3, suppose that 
$N$ has no minor in $\{\Utfi,\Uthfi, Y_3\}$. Clearly, we may also assume that $N$ is simple and $2$-connected and 
$r(N) \ge 2$. By Tutte's Wheels-and-Whirls Theorem (see, for example, \cite[Theorem 8.8.4]{James}), every $3$-connected matroid with at least four elements 
has a $U_{2,4}$- or $M(K_4)$-minor. As $N$ has no minor in $\{\Utfi,\Uthfi, Y_3\}$, the only possible $3$-connected 
minor of $N$ with at least four elements is $U_{2,4}$. Consider the canonical tree decomposition $T$ for $N$. 
By Theorem~\ref{Binary}, we may assume that $N$ is non-binary. Then $T$ has  a vertex  labelled by a copy of 
$U_{2,4}$ and every other vertex of $T$ is labelled by a circuit, a cocircuit, or a copy of $U_{2,4}$. Moreover, as 
$N$ is simple, no leaf of $T$ is labelled by a cocircuit. If $T$ has an interior vertex labelled by a cocircuit, then 
this vertex has  neighbors $x$ and $y$ each of which is labelled by a circuit or a copy of $U_{2,4}$. Thus $N$ 
has as a minor the parallel connection of  two copies of $U_{2,3}$, so $N$ has a $Y_3$-minor; a contradiction. 
Now every $U_{2,4}$-labelled vertex of $T$ is a leaf since if a $U_{2,4}$-labelled vertex has two neighbors, 
then each is labelled by a circuit or a copy of $U_{2,4}$, so $N$ has $Y_3$ as a minor. It follows that  $N$ 
is $U_{2,4}$  or $U_{2,4} \oplus_2 U_{2,4}$, or $N$ can be obtained from an  $n$-circuit $C$ by $2$-summing 
on copies of $U_{2,4}$ across $k$ distinct elements of $C$ for some $k$ with $0 \le k \le n$. We deduce that \2 implies \3. 

Finally, suppose  \3 holds. Then one easily checks that $N$ has no minor in $\{\Utfi,\Uthfi, Y_3\}$, so \3 implies \2, and the theorem holds.
\end{proof}


\begin{thebibliography}{99}

\bibitem{ard}  Ardila, F., The Catalan matroid, {\em J. Combin. Theory Ser. A} {\bf 104} (2003), 49--62.

\bibitem{bik}  Babaioff, M., Immorlica, N., and Kleinberg, R., Matroids, secretary problems, and online mechanisms, {Proceedings of the Eighteenth Annual ACM-SIAM Symposium on Discrete Algorithms}, 434--443, ACM, New York, 2007.


\bibitem{LPSP} Bonin, J. and de Mier, A., Lattice path matroids: Structural properties. {\em European J. Combin.}  \textbf{27} (2006), 701--738.

\bibitem{BDMN} Bonin, J.,  de Mier, A., and   Noy, M.,  Lattice path matroids: enumerative aspects and Tutte polynomials, 
{\it J. Combin. Theory Ser. A} {\bf 104}, (2003) 63--94

\bibitem{CL} Chakraborty, S. and  Lachish, O., 
Improved competitive ratio for the matroid secretary problem, {\it Proceedings of the Twenty-Third Annual ACM-SIAM Symposium on Discrete Algorithms}, 1702--1712, ACM, New York, 2012.

\bibitem{crapo} Crapo, H., 
Single-element extensions of matroids.
{\it J. Res. Nat. Bur. Standards Sect. B} {\bf 69B} (1965),  55--65.

\bibitem{CS} Crapo, H. and Schmitt, W.,  
A free subalgebra of the algebra on matroids, 
{\it European J. Combin.} {\bf 26} (2005),  1066--1085.

\bibitem{cunn} Cunningham, W., {\it A combinatorial decomposition theory}, Ph.D. thesis, University of Waterloo, 1973. 

\bibitem{FSZ} Feldman, M.,  Svensson, O., and  Zenklusen, R.,  A simple $O(log log (rank))$-competitive algorithm for the matroid secretary problem,  {\it Proceedings of the Twenty-Sixth Annual ACM-SIAM Symposium on Discrete Algorithms}, 1189--1201,  SIAM, Philadelphia, PA, 2015.


\bibitem{LF} Finkelstein, L., Two algorithms for the matroid secretary problem, Master's thesis, Technion, Israel Institute of Technology, Haifa, 2011.

\bibitem{GK} Gabow, H. and Kohno, T.,  A network-flow-based scheduler: design, performance history and experimental analysis,
Second Workshop on Algorithmic Engineering and Experiments (ALENEX 00) (San Francisco, CA), 
{\it ACM J. Exp. Algorithmics} {\bf 6} (2001), 1--14. 

\bibitem{huynh} Huynh, T., The matroid secretary problem, The Matroid Union, Posted on September 7, 2015, 
\url{http://matroidunion.org/?p=1487}

\bibitem{IW} Im, S. and Wang, Y., Secretary problems: laminar matroid and interval scheduling, {\it Proceedings of the Twenty-Second Annual ACM-SIAM Symposium on Discrete Algorithms}, 1265--1274, SIAM, Philadelphia, PA, 2011. 

\bibitem{JSZ} Jaillet, P.,  Soto, J., and Zenklusen, R.,  Advances on matroid secretary problems: free order model and laminar case, {\it Integer programming and combinatorial optimization}, 254--265, Lecture Notes in Comput. Sci., 7801, Springer, Heidelberg, 2013.

\bibitem{kl} Kaparis, K. and Letchford, A., On laminar matroids and $b$-matchings, submitted.

\bibitem{lo} Lachish, O., $O(\log \log rank)$ competitive ratio for the matroid secretary problem,  {\it 55th Annual IEEE Symposium on Foundations of Computer Science--FOCS 2014}, 326--335, IEEE Computer Soc., Los Alamitos, CA, 2014.  

\bibitem{James} Oxley, J., {\it Matroid Theory,} Second edition, Oxford University Press, New York, 2011. 


\bibitem{Don} Oxley, J., Prendergast, K., and Row, D., Matroids whose ground sets are domains of functions, {\em J. Austral. Math. Soc. Ser. A} {\bf 32} (1982), 380--387.

\bibitem{PifWel} Piff, M. J. and Welsh, D. J. A., On the vector representation of matroids, {\em J. Comb. Theory Ser. B} {\bf 15} (1973), 51--68.

\bibitem{soh}  Sohoni, M., Rapid mixing of some linear matroids and other combinatorial objects, {\em Graphs Combin.} {\bf 15} (1999), 93--107. 


\bibitem{Soto} Soto, J., Matroid secretary problem in the random assignment model, {\it SIAM J. Comput.} {\bf 42} (2013),   178--211. 

\end{thebibliography}
\end{document}